\documentclass[10pt]{amsart}

\usepackage{amsmath}
\usepackage{amssymb}
\usepackage{amsfonts}

\newtheorem{lemma}{Lemma}
\newtheorem{theorem}{Theorem}

\newtheorem{corollary}{Corollary}
\newtheorem{proposition}{Proposition}

\title[Self-dual binary quadratic operads]{Self-dual binary quadratic operads}
\author{P.S. Kolesnikov}%
\email{pavelsk@math.nsc.ru}%

\begin{document}

\begin{abstract}
We describe those binary quadratic operads generated by 
a~two-dimensional space that are isomorphic to their 
Koszul dual operads.

\end{abstract}

\maketitle

\section{Introduction} 

Varieties of linear algebras defined by polylinear identities of degree 3
(e.g., associative, alternative, Novikov, Poisson, et al.) 
give rise to corresponding binary quadratic operads 
\cite{GK1994, LodVal08}.
The most common case in practice is related with varieties of algebras with 
one binary operation. 
The corresponding operads $\mathcal P$ are generated by 
1- or 2-dimensional $S_2$-module $V(2)=\mathcal P(2)$.
In this paper, we solve the following natural question: 
which binary quadratic operads
$\mathcal P$
governing varieties of algebras with one binary operation
are isomorphic to their Koszul dual operads $\mathcal P^!$? 

For example, it is well known that the operads of associative algebras $\mathrm{As}$
coincides with $\mathrm{As}^!$, the same holds for the operad of Poisson algebras 
$\mathrm{Pois}$. For the operad of (left) Novikov algebras $\mathrm{Nov}$, it is known
that $\mathrm{Nov}^!=\mathrm{Nov}^{op}$, the opposite operad of right Novikov 
algebras \cite{Dzh2011}.

For $\dim V(2)=1,2$, the isomorphism between $\mathcal P$ and $\mathcal P^!$ 
is possible only if $S_2$-module $V(2)$ is isomorphic to its 
skew transpose dual $V(2)^\vee$, i.e., the decomposition 
of $V(2)$ into irreducible $S_2$-modules is $ M_+\oplus M_-$, 
where $M_{\pm }=\Bbbk u_\pm $ are 1-dimensional spaces with 
$(12)u_\pm =\pm u_{\pm}$.

We will explicitly describe {\em self-dual} operads $\mathcal P$, 
i.e., those isomorphic to $\mathcal P^!$. 
First, we split the operads that include self-dual ones into two disjoint classes.
Operads of the first and second class are in one-to-one correspondence with 
the points of the Grassmannian $G(2,4)\subset \mathbb P^5$ and of $G(2,4)\times S(1,1)$, 
respectively. Here $S(1,1)\subset \mathbb P^3$ 
is the Segre variety representing $\mathbb P^1\times \mathbb P^1$.
Next, we determine equations defining explicitly those points in these varieties 
that correspond to self-dual operads. These equations turn to be linear (modulo 
quadratic relations defining Grassmannian and Segre variety).

In this paper, the base field $\Bbbk $ is of characteristic $\ne 2,3$.

\section{Binary quadratic operads}
In this section, we mainly follow notations of \cite{GK1994}, where all necessary 
definitions may be found. 

Let $\mathcal F$ be the free (symmetric) operad generated by a $\mathbb Z_+$-graded space 
$V = \bigoplus_{n\ge 1} V(n)$, where $V(n)=0 $ for $n\ne 2$. 
(By the definition of an operad, $\dim \mathcal F(1)=1$ since it necessarily contains
the identity.)
We will consider the 
most common case when $\dim V(2)=2$, and the symmetric group $S_2$ acts on $V(2)$
by permutation of coordinates: $V(2)=\mathrm{span}\,(\mu, \mu' = (12)\mu )$. 
Then 
\[
 \mathcal F(3) = \Bbbk S_3\otimes _{\Bbbk S_2} (V(2)\otimes V(2)), 
\]
where $(12)\in S_2$ acts on $V(2)\otimes V(2)$ as $\mathrm{id}\otimes (12)$.

Suppose $U$ is an $S_3$-invariant subspace of $\mathcal F(3)$, and let $I=I(U)$ 
be the operad ideal of $\mathcal F$ generated by $U$. 
The quotient operad $\mathcal P = \mathcal F/I$ is denoted by $\mathcal P(V,U)$. 

Given an operad $\mathcal P = \mathcal P(V,U)$, its Koszul dual operad $\mathcal P^!$
is defined as $\mathcal P(V^\vee, U^\perp)$, 
where 
$U^\perp \subseteq \mathcal F(3)^\vee \simeq \Bbbk S_3\otimes _{\Bbbk S_2} (V(2)^\vee \otimes V(2)^\vee )$
is the orthogonal complement of~$U$. 

Let us choose the following linear basis of $\mathcal F(3)$:
\[
\begin{gathered}
e_1 =  1\otimes_{\Bbbk S_2}(\mu\otimes \mu), \quad
e_2 =  (12)e_1, \quad
e_3 =  1\otimes_{\Bbbk S_2}(\mu'\otimes \mu ), \quad
e_4 =  (12)e_3, \\
e_{4+i}  = (13)e_i, \quad e_{8+i}  = (23)e_i, \quad i=1,2,3,4.
\end{gathered}
\]
Fix  a linear basis $f_1,\dots, f_{12}$ in $\mathcal F(3)^\vee$ constructed in the same way 
as $e_1,\dots , e_{12}$
with $\nu $ instead of $\mu $. Denote by $\Sigma $ the ($12\times 12$)-matrix with entries 
$\langle f_i,e_j\rangle $. It is easy to compute that
\[
 \Sigma = \mathrm{diag}\,(1,-1,-1,1,-1,1,1,-1,-1,1,1,-1).
\]

The following statement is a standard exercise in the theory of 
characters.

\begin{lemma}[c.f. \cite{Mal1950}]\label{lem:Lem1}
The $S_3$-module $\mathcal F(3)$ has the following decomposition into 
irreducible $S_3$-modules:
\[
 \mathcal F(3) = 2M_+ \oplus 2M_-\oplus 4 M_2,
\]
where $M_+$ is the trivial module, $M_-$ is the sign module, 
$M_2 = \pi(2,1) $ is the 2-dimensional irreducible $S_3$-module 
determined by a partition $(2,1)$.
\end{lemma}

Let us determine how $M_\pm $ and $M_2$ may be embedded into $\mathcal F(3)$.
All coordinates below are in the basis $e_1,\dots, e_{12}$.

\begin{lemma}[c.f. \cite{Mal1950}]\label{lem:Lem2}
$(1)$ An $S_3$-submodule of $\mathcal F(3)$ isomorphic to $M_\pm$ 
is spanned by a vector 
\[
 u_\pm (x_1,x_2) = (x_1, \pm x_1, x_2, \pm x_2, \pm x_1, x_1, \pm x_2, x_2, \pm x_1, x_1, \pm x_2, x_2),
\]
$(x_1,x_2)\in \Bbbk^2$.

$(2)$ An $S_3$-submodule of $\mathcal F(3)$ isomorphic to $M_2$ 
is spanned by vectors $u_2^i (\bar x)$, $i=1,2$, 
$\bar x = (x_1,x_2,x_3,x_4)\in \Bbbk^4$.
where 
\[
 \begin{aligned}
  u_2^1 = & (x_1, -x_1, x_2, -x_2, x_3, x_3-x_1, x_4, x_4-x_2, x_1-x_3, -x_3, x_2-x_4, -x_4), \\
  u_2^2 = & (x_1-x_3, -x_3, x_2-x_4, -x_4, x_3-x_1, x_3, x_4-x_2, x_4, x_1, -x_1, x_2, -x_2).
 \end{aligned}
\]

\end{lemma}

\begin{proof}
Let us prove (2) for example. 
The only 2-dimensional irreducible $S_3$-module $M_2$
has the following structure: 
\begin{equation}\label{eq:M2Table}
M_2=\Bbbk u_1\oplus \Bbbk u_2,\quad 
\begin{aligned}
  (13):{}& u_1\mapsto u_1-u_2, \\
         & u_2\mapsto -u_2, 
\end{aligned}
\quad
\begin{aligned}
  (23):{}& u_1 \mapsto u_2, \\
         & u_2 \mapsto u_1.
\end{aligned} 
\end{equation}
Suppose the images of $u_i$ in $\mathcal F(3)$ are $u_2^i$, $i=1,2$,
$u_2^1 = x_1e_1+\dots +x_{12}e_{12}$, $u_2^2 = y_1e_1+\dots + y_{12}e_{12}$. 
Then \eqref{eq:M2Table} implies a system of linear equations on $x_i,y_i$, $i=1,\dots, 12$, 
its solution provides the desired statement. 
\end{proof}

\section{Necessary and sufficient conditions of self-duality}
Suppose an operad $\mathcal P = \mathcal P(V,U)$ is isomorphic to its Koszul dual. 
Then $\dim U = \dim U^\perp =6$.
An isomorphism of 
binary operads $\mathcal P$ and $\mathcal P^!$ is determined by an $S_2$-invariant isomorphism of linear spaces 
$\mathcal P(2)=V(2)$ and $\mathcal P^!(2)=V(2)^\vee $. Such an isomorphism is determined by a map 
\[
 g(a,b): \mu \mapsto a \nu + b \nu', 
\]
where $\langle \nu,\mu\rangle =1$, $\langle \nu,\mu'\rangle =0$, 
$\nu' = (12)\nu$, $a,b\in \Bbbk $. Since 
$g(a,b): \mu' \mapsto (12)(a\nu + b\nu') = b\nu + a\nu'$, it is necessary to assume 
$a^2-b^2\ne 0$.

The linear map $g(a,b)$ induces a linear map 
$\Gamma (a,b)=\mathrm{id}\otimes _{\Bbbk S_2} (g(a,b)\otimes g(a,b)): \mathcal F(3)\to \mathcal F(3)^\vee$.
The matrix of $\Gamma (a,b)$ with respect to the bases $e_1,\dots, e_{12}$ and $f_1,\dots, f_{12}$
is also denoted by $\Gamma (a,b)$. It is easy to compute that
\[
\Gamma (a,b) = 
\begin{pmatrix}
 A & 0 & 0\\
 0&A&0\\
 0&0&A
\end{pmatrix},
\quad 
A =
\begin{pmatrix}
  a^2 & ab & ab & b^2 \\
  ab & a^2 & b^2 & ab \\
  ab & b^2 & a^2 & ab \\
  b^2& ab & ab & a^2 
\end{pmatrix},
\]
where 
\[
 \Gamma (a,b): \bar e \mapsto \Gamma (a,b)\bar f, 
 \quad 
 \bar e=\begin{pmatrix}
 e_1 \\ \vdots \\ e_{12}
 \end{pmatrix},
\quad 
\bar f=\begin{pmatrix}
f_1 \\ \vdots \\ f_{12}
\end{pmatrix}.
\]

Obviously, $\mathcal P\simeq \mathcal P^!$ if and only if there exist $a,b\in \Bbbk $, $a^2\ne b^2$, 
such that 
\begin{equation}\label{eq:Orth}
   \Gamma (a,b)U = U^\perp . 
\end{equation}
The linear space $U$ may be presented by a $(6\times 12)$-matrix 
(also denoted by $U$) of coordinates with respect to the basis 
$e_1,\dots, e_{12}$. Then \eqref{eq:Orth} may be rewritten as
\[
0= \big\langle U\Gamma (a,b)\bar f, (U\bar e)^T \big\rangle = U\Gamma (a,b) \langle \bar f,\bar e^T\rangle U^T
=U\Gamma (a,b)\Sigma U^T,  
\]
assuming that $\mathrm{rank}\,U=6$. Moreover, the matrix $U$ should determine an $S_3$-invariant subspace 
of $\mathcal F(3)$.
The latter condition together with Lemma \ref{lem:Lem1} provides the following opportunities 
for the space $U$ as for an $S_3$-module:
\begin{itemize}
 \item[(R1)] $U\simeq 2M_+ \oplus 2M_- \oplus M_2$;
 \item[(R2)] $U\simeq 3M_2$;
 \item[(R3)] $U\simeq 2M_+ \oplus M_2$;
 \item[(R4)] $U\simeq 2M_- \oplus M_2$;
 \item[(R5)] $U\simeq M_+ \oplus M_- \oplus 2M_2$.
\end{itemize}
We will consider the cases (R1)--(R5) separately.

For $u,v\in \Bbbk^{12}$, denote 
\[
 \langle u, v\rangle_{a,b} = u\Gamma (a,b)\Sigma v^T \in \Bbbk .
\]

\begin{lemma}\label{lem:Lem3Orth}
Let $x_i,y_i\in \Bbbk $, $i=1,\dots, 4$, $a,b\in \Bbbk $. Then \\
$(1)$ $\langle u_+(x_1,x_2), u_+(y_1,y_2) \rangle_{a,b} =0 $;\\
$(2)$ $\langle u_-(x_1,x_2), u_-(y_1,y_2) \rangle_{a,b} =0 $;\\
$(3)$ $\langle u_\pm (x_1,x_2), u_2^i(y_1,y_2,y_3,y_4) \rangle_{a,b} =0 $, $i=1,2$;\\
$(4)$ $\langle u_2^i(x_1,x_2,x_3,x_4), u_\pm (y_1,y_2) \rangle_{a,b} =0 $, $i=1,2$;\\
$(5)$ $\langle u_2^i(x_1,x_2,x_3,x_4), u_2^i (y_1,y_2,y_3,y_4) \rangle_{a,b} =0 $, $i=1,2$.
\end{lemma}

\begin{proof}
 Straightforward computation.
\end{proof}
 
\begin{lemma}\label{lem:Lem4PM}
 Let $x_i,y_i\in \Bbbk $, $i=1,2$, $a,b\in \Bbbk $, $a^2\ne b^2$. 
 Then the following conditions are equivalent:
 \begin{itemize}
  \item $\langle u_+(x_1,x_2),u_-(y_1,y_2)\rangle_{a,b}  = \langle u_-(y_1,y_2),u_+(x_1,x_2)\rangle_{a,b} =0 $;
  \item $a(x_1y_1-x_2y_2) + b(x_2y_1-x_1y_2)=0$.
 \end{itemize}
\end{lemma}

\begin{proof}
 Straightforward computation.
\end{proof}
 
\begin{corollary}
If $\mathcal P = \mathcal P(V,U) $ is isomorphic to its Koszul dual operad $\mathcal P^!$ 
then $U$ may not have a representation type (R1) or (R2).
\end{corollary}

\begin{proof}
If $U$ is of type (R1) then $u_\pm (x_1,x_2)\in U$ for all $x_1,x_2\in \Bbbk $.
Lemma~\ref{lem:Lem4PM} implies $\langle U,U\rangle _{a,b}=0$ is impossible.

For further needs, let us introduce matrices $T_i\in M_{4,12}(\Bbbk )$, $i=1,2$,
 such that $u_2^i(x_1,x_2,x_3,x_4) = ( x_1\ x_2 \ x_3 \ x_4) T_i$. 
 Then
 \[
  \langle u_2^i(x_1,x_2,x_3,x_4), u_2^j (y_1,y_2,y_3,y_4) \rangle_{a,b} = 
  ( x_1 \ x_2 \ x_3 \ x_4 ) A_{ij}(a,b)
  \begin{pmatrix} y_1 \\ y_2 \\ y_3 \\ y_4 \end{pmatrix},
 \]
where 
$A_{ij}(a,b) = T_i \Gamma (a,b)\Sigma T_j^T$.
It is easy to check that $A_{ij}(a,b)=-A_{ji}(a,b)$. 

If $U$ is of type (R2) then there exist a matrix $\widehat U \in M_{3,4}(\Bbbk )$ of rank 3 such that 
\begin{equation}\label{eq:SDUcond}
\widehat U A_{12}(a,b) \widehat U^T =0, 
\end{equation}
 this condition is necessary and sufficient for self-duality 
of $\mathcal P(V,U)$ by Lemma~\ref{lem:Lem3Orth}(5). Up to row transformations, $\widehat U$ may be chosen in one of the 
following forms: 
\[
         \begin{pmatrix}
               t_1 & 1 & 0 & 0 \\
               t_2 & 0 & 1 & 0 \\
               t_3 & 0 & 0 & 1
              \end{pmatrix},
\
       \begin{pmatrix}
             1 & t_1 & 0 & 0 \\
             0 & t_2 & 1 & 0 \\
             0 & t_3 & 0 & 1
              \end{pmatrix},
\
       \begin{pmatrix}
             1 & 0 & t_1 & 0 \\
             0 & 1 & t_2 & 0 \\
             0 & 0 & t_3 & 1
              \end{pmatrix},
\ 
       \begin{pmatrix}
             1 & 0 & 0 & t_1 \\
             0 & 1 & 0 & t_2 \\
             0 & 0 & 1 & t_3
              \end{pmatrix}.
\]
For either of these $\widehat U$, \eqref{eq:SDUcond} turns into a system of algebraic equations on $t_1,t_2,t_3$ with parameters $a,b\in \Bbbk $. 
All four systems obtained are incompatible with the condition $a^2-b^2\ne 0$: we applied 
computer algebra system {\tt Singular} \cite{Singular}
to compute a Gr\"obner basis of \eqref{eq:SDUcond} 
together with additional relation $(a^2-b^2)c-1$ for a new variable~$c$ to find all these ideals 
to be improper.
\end{proof}

\begin{proposition}\label{prop:Dim2Cases}
 Suppose $\widehat U$ be a $(2\times 4)$-matrix of rank 2 satisfying 
 \eqref{eq:SDUcond} for some $a,b\in \Bbbk $ such that $a^2\ne b^2$.
Then, up to row transformations, $\widehat U$ is one of the following:
\begin{enumerate}
\item[(U1)]
\[
   \widehat U = \begin{pmatrix}
                 \gamma & \gamma & 0 & 1 \\
                 -\gamma & -\gamma & 1 & 0
                \end{pmatrix},
                \ \gamma \in \Bbbk , \quad 
  \widehat U = \begin{pmatrix}
                 0 & 0 &  1 & 1 \\
                 1 & 1 & 0 & 0
                \end{pmatrix},              
\]
relation \eqref{eq:SDUcond} holds for every $a,b\in \Bbbk $;
\item[(U2)] 
\[
   \widehat U = \begin{pmatrix}
                 -\gamma & \gamma+2 & 0 & 1 \\
                 2-\gamma & \gamma & 1 & 0
                \end{pmatrix},
\ \gamma \in \Bbbk , \quad 
   \widehat U = \begin{pmatrix}
                 0 & 0 & - 1 & 1 \\
                 - 1 & 1 & 0 & 0
                \end{pmatrix},
\]
relation  \eqref{eq:SDUcond} holds for every $a,b\in \Bbbk $;
 
\item[(U3)]
\begin{equation}\label{eq:Ugeneric}
   \widehat U = \begin{pmatrix}
                 x_1 & x_2 & x_3 & x_4 \\
                 y_1 & y_2 & y_3 & y_4
               \end{pmatrix},
\end{equation}
where 
\[
x_1y_3 - x_3y_1 +x_4y_2-x_2y_4 = 0,
\]
relation \eqref{eq:SDUcond} holds for $b=0$;

\item[(U4)]
$\widehat U$ as in \eqref{eq:Ugeneric}, satisfying 
\[
x_2y_1 - x_1y_2 +x_3y_2-x_2y_3+x_1y_4-x_4y_1 = 0,
\]
relation  \eqref{eq:SDUcond} holds for $a=0$.
\end{enumerate}
\end{proposition}

\begin{proof}
Case 1. Suppose $\widehat U$ satisfies \eqref{eq:SDUcond} for some $a\ne 0$, $b\ne 0$, 
and $\widehat U$ may be transformed to 
\[
 \begin{pmatrix}
  t_1 & t_2 & 0 & 1 \\
  t_3 & t_4 & 1 & 0
 \end{pmatrix}.
\]
Consider the coefficients of 
$\widehat U A_{12}(a,b) \widehat U^T$ as polynomials in 
variables $t_1,\dots, t_4$, $a$, $b$, $a^{-1}$, $b^{-1}$, $c$, and 
compute Gr\"obner basis of the ideal generated by these coefficients 
together with $aa^{-1}-1$, $bb^{-1}-1$, $(a^2-b^2)c-1$.
Gr\"obner basis obtained consists of some polynomials in $t_1,\dots, t_4$ 
and in $a,b,a^{-1},b^{-1},c$ only. Polynomials in $t_1,\dots, t_4$ are
\[
\begin{aligned}
& t_3^2-t_4^2-2t_3 +2t_4, \\
& t_2^2-t_1^2+2t_1-2t_2, \\
& t_2t_4 + t_3t_4 -t_2+t_3-2t_4, \\
& t_2t_3+t_4^2-t_2-t_3, \\
& t_1+t_4.
\end{aligned}
\]
Every solution of this system provides a matrix $\widehat U$ 
which satisfies \eqref{eq:SDUcond} for all $a,b\in \Bbbk$.
The first two polynomials split into linear factors:
\[
 (t_3-t_4)(t_3+t_4-2),\quad (t_2-t_1)(t_2+t_1-2).
\]
This provides four subcases.

Case 1.1: $t_1=t_2$, $t_3=t_4$. Then $t_1=t_2=-t_3=-t_4$ as described in (U1).

Case 1.2: $t_1+t_2=2$, $t_3+t_4=2$. Then $t_1=2-t_2=-t_4$, $t_3=2-t_4$ as described in (U2). 

Case 1.3: $t_1=t_2$, $t_3+t_4=2$. Then $t_4=\pm 1$, $t_2=t_1=-t_4$, $t_3=2-t_4$.
If $t_4=1$ then we obtain $\widehat U$ as in Case 1.1. For $t_4=-1$, this $\widehat U$
is described by Case 1.2.

Case 1.4: $t_1+t_2=2$, $t_3=t_4$. Then $t_4=\pm 1$, $t_2=2+t_4$, $t_1=-t_4$, $t_3=t_4$. 
Both these matrices are already described in Cases 1.1 and 1.2.

Case 2. Suppose $\widehat U$ satisfies \eqref{eq:SDUcond} for some $a\ne 0$, $b\ne 0$, 
and $\widehat U$ may be transformed to 
\[
 \begin{pmatrix}
  t_1 & 0 & t_2 & 1 \\
  t_3 & 1 & t_4 &  0
 \end{pmatrix}.
\]
It makes sense to consider this case for $t_4=0$ only since for $t_4\ne 0$ this matrix may 
be transformed by row transformations to a matrix from Case 1.
As in Case 1, we obtain the following solution:
\[
t_1=0,\quad t_2=t_3=\pm 1,\quad t_4=0.
\]
These are the matrices described in (U1) and (U2) for $t_3=1$ and $t_3=-1$, respectively.

Case 3. 
Suppose $\widehat U$ satisfies \eqref{eq:SDUcond} for some $a\ne 0$, $b\ne 0$, 
and $\widehat U$ may be transformed to either of 
\[
 \begin{pmatrix}
  0 & t_1 & t_2 & 1 \\
  1 & t_3 & t_4 & 0
 \end{pmatrix},
\quad 
  \begin{pmatrix}
  t_1 & 0 & 1 & t_2  \\
  t_3 & 1 & 0 & t_4 
 \end{pmatrix},
\quad 
 \begin{pmatrix}
  0 & t_1 & 1 & t_2 \\
  1 & t_3 & 0 & t_4
 \end{pmatrix}.
\]
It is enough to consider these matrices for $t_4=0$, they do not produce new solutions 
of \eqref{eq:SDUcond}.

Case 4. 
Suppose $\widehat U$ satisfies \eqref{eq:SDUcond} for some $a\ne 0$, $b\ne 0$, 
and $\widehat U$ may be transformed to 
\[
 \begin{pmatrix}
  0 & 1 & t_1 & t_2 \\
  1 & 0 & t_3 & t_4
 \end{pmatrix}.
\]
This case makes sense to consider for $t_1=t_2=t_3=t_4=0$, this does not produce 
solutions of \eqref{eq:SDUcond}.

Case 5.
Suppose $\widehat U$ satisfies \eqref{eq:SDUcond} for $ab = 0$.
Then 
\[
 \widehat U (A_{12}(a,b) \pm A_{12}(a,b)^T) \widehat U^T = 0. 
\]
Straightforward computation shows $A_{12}+A_{12}^T=0$. 

Case 5.1: $b=0$, $a\ne 0$. Then 
$\widehat U (A_{12} - A_{12}^T) \widehat U^T = 0$
if and only if 
$x_1y_3-x_3y_1+x_4y_2-x_2y_4=0$ as in (U3).

Case 5.1: $a=0$, $b\ne 0$. Then 
$\widehat U (A_{12} - A_{12}^T) \widehat U^T = 0$
if and only if 
$x_2y_1-x_1y_2 +x_3y_2-x_2y_3+x_1y_4-x_4y_1=0$ as in (U4).
\end{proof}

Every $2\times 4$ matrix $\widehat U$ in Proposition~\ref{prop:Dim2Cases}
gives rise to a subspace of $\Bbbk ^4$ generated by its rows, so this statement actually describes 
a subset of the Grassmannian $G = G(2,4)\subset \mathbb P^5$.
Every point of $G$ is defined by Pl\"ucker coordinates 
$(p_{ij})_{1\le i<j\le 4}$: 
a subspace spanned by $\widehat U$ of the form \eqref{eq:Ugeneric}
has coordinates $p_{ij}=x_iy_j-x_jy_i$.
Recall that $p_{ij}$ satisfy Pl\"ucker identity 
$p_{12}p_{34}-p_{13}p_{24}+p_{14}p_{23}=0$.

The following statement is an easy exercise in linear algebra.

\begin{lemma}\label{rem:Dim2Cases}
The following equations determine subspaces generated by rows of 
matrices in Proposition \ref{prop:Dim2Cases}.
\begin{itemize}
 \item [(U1)] $p_{12}=0$, $p_{13}=p_{14}=p_{23}=p_{24}$;
 \item [(U2)] $p_{13}=p_{24}$, $p_{14}+p_{23}+2p_{13}=0$, $p_{12}=4p_{34}$, $p_{14} = p_{23}+p_{12}$;
 \item [(U3)] $p_{13}=p_{24}$;
 \item [(U4)] $p_{14} = p_{12}+p_{23}$.
\end{itemize}
\end{lemma}

\subsection{Description of self-dual operads}

Suppose $U\subset \mathcal F(3)$ is an $S_3$-submodule isomorphic to $2M_\pm \oplus 2M_2$. 
Then $U$ is is spanned by 
\[
 u_\pm(1,0),\ u_\pm(0,1), \ u_2^i(\bar x), \ u_2^i(\bar y),\ i=1,2, 
\]
where $\bar x, \bar y\in \Bbbk ^4$ span a 2-dimensional subspace $\widehat U$ of $\Bbbk ^4$ 
which may be identified with a point $p=p(U)$ of the Grassmannian $G=G(2,4)$.

If $U\subset \mathcal F(3)$ is an $S_3$-submodule isomorphic to $M_+\oplus M_- \oplus 2M_2$
then $U$ is is spanned by 
\[
 u_+(s_1,s_2),\ u_-(t_1,t_2), \ u_2^i(\bar x), \ u_2^i(\bar y),\ i=1,2, 
\]
where $\bar x, \bar y\in \Bbbk ^4$ are encoded by the Pl\"ucker coordinates 
$(p_{ij})_{1\le i<j\le 4}$ of $\mathrm {span}\,(\bar x,\bar y)\in G$ as above,
and $(s_1,s_2, t_1,t_2)$ are encoded by a point in the Segre variety $S = S(1,1) \subset \mathbb P^3$, 
$S\simeq \mathbb P^1\times \mathbb P^1$. Such a point is given by coordinates $(z_{kl})_{1\le k,l\le 2}$, 
$z_{kl}=s_kt_l-s_lt_k$, satisfying $z_{11}z_{22}-z_{12}z_{21} =0$. 
Therefore, such 
subspaces $U$ are in one-to-one correspondence with points $q(U) \in G\times S$ with 
coordinates $(p_{ij}, z_{kl})_{1\le i<j\le 4, k,l=1,2}$.

Let $\mathcal P = \mathcal P(V,U)$ be a binary quadratic operad generated by
$V=V(2)$ which is is isomorphic to $M_+\oplus M_-$ as an $S_2$-module.
Assume $\mathcal P\simeq \mathcal P^!$.
Then, as we have already seen, $U$ as an $S_3$-module is isomorphic either to 
$2M_\pm \oplus 2M_2$ or to $M_+\oplus M_-\oplus 2M_2$.
In the first case, $U$ may be identified with a point in $G$ as above, 
in the second case $U$ is encoded by a point of $G\times S$. 
The following statements describe these points explicitly.

\begin{theorem}\label{thm:2+2case}
Let  $U\simeq 2M_\pm \oplus 2M_2$ 
as an $S_3$-module. 
Then $\mathcal P=\mathcal P(V,U)$ is isomorphic to 
$\mathcal P^!$ if and only if 
$p(U)\in Y_1\cup Y_2\in G=G(2,4)$, 
where 
\[
 \begin{aligned}
 Y_1 & = \{ (p_{ij})\in G \mid   p_{13}=p_{24} \}, \\
 Y_2 & = \{ (p_{ij})\in G \mid   p_{14}=p_{12}+p_{23} \}. \\
 \end{aligned}
\]
\end{theorem}

\begin{proof}
Follows immediately from  
Lemma \ref{lem:Lem3Orth} and Proposition~\ref{prop:Dim2Cases}.
Note that $p(U)\in Y_1$ if and only if $\mathcal P=\mathcal P^!$, 
and $p(U)\in Y_2$ if and only if $\mathcal P^!=\mathcal P^{op}$.
\end{proof}

\noindent
{\bf Example.}
For the operad Nov governing the variety of Novikov algebras, i.e.,
right-symmetric left commutative algebras, the space 
of defining relations $U\subset \mathcal F(3)$ 
is isomorphic to $2M_-\oplus 2M_2$.
The corresponding point $p(U)$ of the Grassmannian  belongs to 
$Y_2$ 
 with
$p_{12}=-1$,
$p_{13}=0$,
$p_{14}=1$,
$p_{23}=2$,
$p_{24}=3$, 
$p_{34}=2$. 

\begin{theorem}\label{thm:1+1+2case}
Let $U\simeq M_+\oplus M_- \oplus 2M_2$  
as an $S_3$-module. 
Then $\mathcal P=\mathcal P(V,U)$ is isomorphic to 
$\mathcal P^!$ if and only if 
$q(U)\in X_1\cup X_2\cup X_3\cup X_4 \subset G\times S$, 
where 
\[
\begin{aligned} 
 X_1 & = \{ (p_{ij}, z_{kl}) \in G\times S \mid && p_{12}=0, p_{13}=p_{14}=p_{23}=p_{24},\\
     & && z_{11}^2 + z_{22}^2 \ne z_{12}^2+z_{21}^2\}, \\
 X_2 & =  \{ (p_{ij}, z_{kl}) \in G\times S  \mid &&  p_{13}-p_{24}=p_{14}+p_{23}+2p_{13} =0, \\
     & && p_{12}-4p_{34}= p_{14}-p_{12}-p_{23}=0,\\
     & && z_{11}^2 + z_{22}^2 \ne z_{12}^2+z_{21}^2\}, \\
 X_3 & =  \{ (p_{ij}, z_{kl}) \in G\times S \mid && p_{13}=p_{24},  z_{11}=z_{22} \}, \\
 X_4 & =  \{ (p_{ij}, z_{kl}) \in G\times S \mid && p_{14}=p_{12}+p_{23},  z_{12}=z_{21} \}.
\end{aligned}
\] 
\end{theorem}

\begin{proof}
Proposition \ref{prop:Dim2Cases} together with Lemma \ref{rem:Dim2Cases}
describe necessary condi\-ti\-ons on $\widehat U= \mathrm{span}\,(\bar x,\bar y)\in G(2,4)$.
It remains to determine $s_i$ and $t_i$, $i=1,2$, in such a way that 
$\langle u_+(s_1,s_2), u_-(t_1,t_2)\rangle _{a,b} =0$
for some $a,b\in \Bbbk$, $a^2\ne b^2$. 
By Lemma~\ref{lem:Lem4PM}, orthogonality condition for $u_+$ and $u_-$ 
is equivalent to $a(z_{11}-z_{22})=b(z_{12}-z_{21})$. 
If $ab\ne 0$ then 
$a^2\ne b^2$ implies
\begin{equation}\label{eq:S_0Equation}
 z_{11}^2 + z_{22}^2 \ne z_{12}^2+z_{21}^2.
\end{equation}
If $\widehat U$ is given by (U1) or (U2) of Proposition \ref{prop:Dim2Cases}
then the only condition on $s_i$, $t_i$ is given by \eqref{eq:S_0Equation}. 
If $\widehat U$ is given by (U3) or (U4) of Proposition \ref{prop:Dim2Cases}
then $b=0$ or $a=0$, respectively. In each of these cases, the second parameter may be 
chosen to be nonzero if and only if $z_{11}-z_{22}=0$ or $z_{12}-z_{21}=0$, respectively.
\end{proof}

\noindent
{\bf Example.}
 The operad As governing the variety of associative algebras belongs to 
$X_3\cap X_4$ 
 with $z_{11}=z_{22}=1$, $z_{12}=z_{21}=-1$, 
$p_{12}=p_{13}=p_{14}=p_{24}=p_{34}=1$, $p_{23}=0$. 

\noindent
{\bf Example.}
The operad Pois governing the variety of Poisson algebras 
in terms of one operation (see \cite{MR2007}) belongs to 
$X_3\cap X_4$ 
 with $z_{11}=z_{22}=1$, $z_{12}=z_{21}=-1$, 
$p_{12}=0$, $p_{13}=p_{14}=p_{23}=p_{24}=p_{34}=1$.

\end{document}